\documentclass{aims}
\usepackage{amsmath}
\usepackage{enumitem}
\usepackage[pagewise]{lineno}
\usepackage{graphics} 
\usepackage{epsfig} 
\usepackage{graphicx}  \usepackage{epstopdf}
\usepackage[colorlinks=true]{hyperref}
\hypersetup{urlcolor=blue, citecolor=red}

\def\currentvolume{X}
 \def\currentissue{X}
  \def\currentyear{200X}
   \def\currentmonth{XX}
    \def\ppages{X--XX}

\newtheorem{theorem}{Theorem}[section]

\newtheorem{lemma}[theorem]{Lemma}
\newtheorem{proposition}{Proposition}

\theoremstyle{definition}
\newtheorem{definition}[theorem]{Definition}
\newtheorem{remark}{Remark}

\newcommand{\intersection}{\mathop{\bigcap}}
\newcommand{\union}{\mathop{\bigcup}}

\newcommand{\del}{\partial}
\newcommand{\norm}[1]{\left\lVert#1\right\rVert}

\newcommand{\Wloc}{W_{\mathrm{loc}}}
\newcommand{\sgn}{\mathrm{sgn}}

\newcommand{\N}{\mathbb N}
\newcommand{\Z}{\mathbb Z}

\newcommand{\R}{\mathbb R}

\let\phi\varphi

\let\epsilon\varepsilon

\let\tilde\widetilde

\let\hat\widehat

\title[SRB measures of singular hyperbolic attractors] 
      {SRB measures of singular hyperbolic attractors}

\author[Dominic Veconi]{}


 \email{dkv5049@psu.edu}

\begin{document}
\maketitle

\centerline{\scshape Dominic Veconi$^*$}
\medskip
{\footnotesize
 \centerline{Department of Mathematics}
   \centerline{The Pennsylvania State University}
   \centerline{University Park, PA 16802, USA}
} 

%

\bigskip

 \centerline{(Communicated by the associate editor name)}

\begin{abstract}
	It is known that hyperbolic maps admitting singularities have at most countably many ergodic Sinai-Ruelle-Bowen (SRB) measures. These maps include the Belykh attractor, the geometric Lorenz attractor, and more general Lorenz-type systems. In this paper, we establish easily verifiable sufficient conditions guaranteeing that the number of ergodic SRB measures is at most finite, and provide examples and nonexamples showing that the conditions are necessary in general. 
\end{abstract}

\section{Introduction}

One primary question in smooth ergodic theory is the existence of ``physical measures'' for a smooth dynamical system. Given a compact Riemannian manifold $M$ and a smooth map $f : U \to M$, $U \subseteq M$ open, a \emph{physical measure} is one in which the Birkhoff averages of continuous functions are constant on a set of positive measure. In other words, a probability measure $\mu$ is a \emph{physical measure} if
\[
m\bigg\{ x \in U : \lim_{n \to +\infty} \frac 1 n \sum_{k=0}^{n-1} \left( \phi \circ f^k\right)(x) = \int_U \phi \, d\mu \quad \forall \phi \in C^0(U)\bigg\} > 0,
\]
where $m$ is the Riemannian volume. Among the most significant physical measures are the \emph{Sinai-Ruelle-Bowen (SRB) measures}. These are invariant measures for hyperbolic dynamical systems that have conditional measures on unstable leaves that are absolutely continuous with respect to the Riemannian leaf volume. For uniformly hyperbolic dynamical systems (such as transitive Anosov diffeomorphisms and attractors of Axiom A systems), there is a unique SRB measure \cite{Bowen}, and the existence of SRB measures has been established for several classes of nonuniformly hyperbolic dynamical systems \cite{Hu99,PSZ-Katok,me-PAD} and partially hyperbolic dynamical systems \cite{ABV, BV}. It was further shown in \cite{RH2TU} that if $M$ is a compact Riemannian 2-manifold and $f : M \to M$ is a hyperbolic diffeomorphism admitting an SRB measure, then this SRB measure is unique. 

Many dynamical systems in engineering and natural sciences exhibit ``chaotic'' behavior: their trajectories appear disordered and they are highly sensitive to initial data. The simplest mathematical examples of such systems are uniformly hyperbolic and uniformly expanding, and so hyperbolic dynamical systems have been at the forefront of smooth ergodic theory since at least the 1960s. However, most stochastic dynamical systems arising from physical and natural phenomena are not uniformly hyperbolic. In these instances, uniqueness results for uniformly hyperbolic dynamical systems and surface maps (such as those described in \cite{RH2TU}) no longer apply. Examples of such dynamical systems include the Lorenz attractor model of atmospheric convection, the associated geometric Lorenz attractor (described in more detail below), and the Belykh attractor of phase synchronization theory \cite{APPV, PeGHA, Sat10}. The latter two are maps of the unit square that admit highly complex limit sets, so that the resulting maps are not invariant under Lebesgue volume and may not \emph{a priori} admit a unique SRB measure. 

Although our results concern discrete singular hyperbolic maps, historically many results about singular hyperbolic attractors come from investigations into hyperbolic flows, the most famous being the flow generated by the Lorenz equations. In \cite{KL}, J. Kaplan and J. Yorke used a Poincar\'e return map to study the dynamical behavior of the Lorenz attractor, such as the parameters for which periodic points are dense. This Poincar\'e map was later reformulated as the \emph{geometric Lorenz attractor}, which is a simplified discrete model of the Poincar\'e map of the original Lorenz flow. The more general family of discrete \emph{Lorenz-type maps} was introduced in \cite{APLorenz}. In the years that followed, the Lorenz system and related hyperbolic flows have led to active research in singular hyperbolic attractors (see e.g. \cite{APLorenz, APPV, PeGHA, Sat92, Sat10}, and others). There is also a large body of work on singular hyperbolic and sectional-hyperbolic flows more generally. In \cite{Sat10}, it is shown that singular hyperbolic flows admit finitely many ergodic physical measures; more recently, it was shown in \cite{Asectional} that flows of H\"older-$C^1$ vector fields admitting a sectional-hyperbolic attracting set admit finitely many ergodic SRB measures. The proof in \cite{Asectional} also relies on Poincar\'e return maps, and so these results extend to discrete singular hyperbolic maps arising as Poincar\'e maps of hyperbolic flows. For a detailed discussion of the ergodic properties of hyperbolic flows and their attractors, see \cite{APbook}.

In this paper, we consider the class of discrete singular hyperbolic dynamical systems. These are hyperbolic maps $f : K \setminus N \to K$, where $K \subset M$ is a precompact open subset of a Riemannian manifold $M$, and $N \subset K$ is a closed subset of singularities on which $f$ fails to be continuous and/or differentiable. The map $f$ is uniformly hyperbolic on the non-invariant set $K \setminus N$, but behaves more similarly to the non-uniformly hyperbolic setting on an invariant set that consists of trajectories passing nearby the the singular set $N$ with a prescribed rate. Our setting includes systems that are derived from Poincar\'e maps of hyperbolic flows, such as the geometric Lorenz attractor, but also includes singular hyperbolic dynamical systems that do not arise from flows, such as the Lozi map \cite{Lozi}. In \cite{PeGHA}, it was shown that the attractors admitted by singular hyperbolic maps support at most countably many ergodic SRB measures. In \cite{Sat92}, conditions were given under which a singular hyperbolic attractor admits at most finitely many ergodic SRB measures. We provide an alternative proof of this result, with somewhat different conditions that are easy to verify. Namely, if the singular set is a disjoint union of finitely many embedded submanifolds that transversally intersect unstable leaves, and if the image of neighborhoods of the singular set remain separated from the singular set under the dynamics for sufficiently long but finite time (conditions (SH3), (SH6), and (SH7)), then there are at most finitely SRB measures. 


Although most examples of singular hyperbolic attractors in the literature satisfy conditions (SH3) - (SH7) (see the examples in \cite{APLorenz, APPV,Belykh,Lozi,PeGHA}), there exist singular hyperbolic attractors that do not satisfy these conditions and admit infinitely many ergodic components. For this reason, these conditions are necessary for our statement to be true in full generality. For example, one can construct a family of Lorenz-type attractors whose singular sets have infinitely many components, and these examples admit countably many SRB measures, but these maps are not topologically transitive (see Section \ref{infinite-example}). 


Once the existence of ergodic SRB measures has been established, a natural question to ask is when the SRB measure is unique. In \cite{RH2TU}, it is shown that in the case of a $C^{1+\alpha}$ diffeomorphism $f$ of a compact surface, if $f$ is topologically transitive, then $f$ admits at most one SRB measure. A similar result may be proven for singular hyperbolic attractors, provided a certain regularity condition on the stable foliation (Theorem \ref{SRB-uniqueness}, also \cite{PeGHA}). The regularity condition needed is \emph{local continuity}, which roughly means that the smooth functions $E^s(x) \to M$ defining the local stable leaf $\Wloc^s(x)$ vary continuously with $x \in K \setminus N$, where $E^s(x) \subset T_x M$ is the stable subspace at $x$. In this paper, we show that a singular hyperbolic map for which the stable foliation is locally continuous admits a unique SRB measure if and only if the map is topologically transitive. 


This paper is structured as follows. Section \ref{Preliminaries} is devoted to preliminary constructions and definitions needed to discuss singular hyperbolic dynamical systems. Our main result is stated and proven in Section \ref{Main result and proof}. Section \ref{section-examples} is spent discussing examples of dynamical systems satisfying the hypotheses of our main result, as well as examples of systems that fail these hypotheses and that admit infinitey many SRB measures.

\section{Preliminaries}\label{Preliminaries}

We begin by defining singular hyperbolic attractors, and discuss some of their major properties. We consider a Riemannian manifold $M$,  an open, bounded, connected subset $K \subset M$ with compact closure, and a closed subset $N \subset K$. We further consider a map $f : K \setminus N \to K$ satisfying: 
\begin{enumerate}[label=(SH\arabic*),leftmargin=0.5in]
	\item $f$ is a $C^2$ diffeomorphism from $K \setminus N$ to $f(K \setminus N)$.
\end{enumerate}
We further define $N^+ := N \cup \del K$ as the \emph{discontinuity set} for $f$ (on which the function $f$ is discontinuous), and further define
\[
N^- = \left\{y \in K : \textrm{There are } z \in N^+ \textrm{ and } z_n \in K\setminus N^+ \textrm{ s.t. } z_n \to z \textrm{ and } f(z_n) \to y\right\}.
\]
The set $N^-$ is referred to as the \emph{discontinuity set} for $f^{-1}$. We further assume the map $f$ satisfies:
\begin{enumerate}[label=(SH\arabic*),leftmargin=0.5in]
	\setcounter{enumi}{1}
	\item There exist $C_i>0$ and $\alpha_i \geq 0$, with $i = 1,2$, such that
	\begin{align*}
	\norm{d^2 f_x} &\leq C_1 \rho(x, N^+)^{-\alpha_1} \quad \textrm{for } x \in K \setminus N, \\
	\norm{d^2 f_x^{-1}} &\leq C_2 \rho(x, N^-)^{-\alpha_2} \quad \textrm{for }x \in f(K \setminus N)
	\end{align*}
	where $\rho$ is the Riemannian distance in $M$.
\end{enumerate}

Define the set $K^+$ by
\begin{align*}
K^+ = \intersection_{n=0}^\infty \left(K \setminus f^{-n}(N^+)\right) = \left\{ x \in K : f^n(x) \not\in N^+ \textrm{ for all } n \geq 0\right\},
\end{align*}
so that $K^+$ is the largest forward-invariant set on which $f$ is continuous. Further, define
\[
D = \intersection_{n=0}^\infty f^n(K^+) \quad \textrm{and} \quad \Lambda = \overline D.
\]
We say $\Lambda$ is the \emph{attractor} for $f$. 

\begin{proposition}\cite{PeGHA}\label{PeGHA Prop 1}
	We have $D = \Lambda \setminus \union_{n \in \Z} f^n(N^+)$. Furthermore, $f$ and $f^{-1}$ are well-defined on $D$, and $f(D) = D$ and $f^{-1}(D) = D$. 
\end{proposition}

Given $z \in M$, $\alpha > 0$, and a subspace $P \subset T_zM$, we denote the cone at $z$ around $P$ with angle $\alpha$ by
\[
C(z,\alpha,P) = \left\{v \in T_z M : \angle(v,P) := \inf_{w \in P} \angle(v,w) \leq \alpha \right\}.
\]

\begin{definition}\label{main-def}
	The map $f : K\setminus N \to K$ is \emph{singular hyperbolic} if there is $C > 0$, $\lambda > 1$, a function $\alpha : D \to \R$, and two distributions $P^s, P^u$ on $K \setminus N^+$ of dimensions $\dim P^s = p$, $\dim P^u = q = n-p$ (with $n = \dim M$), such that the cones $C^s(z) = C(z,\alpha(z), P^s_z)$ and $C^u(z) = C(z,\alpha(z), P^u_z)$, for $z \in K \setminus N$, satisfy the following conditions: 
	\begin{enumerate}[label=(\alph*)]
		\item The angle between $C^s(z)$ and $C^u(z)$ is uniformly bounded below over $K \setminus N^+$, and in particular, $C^s(z) \cap C^u(z) = 0$; 
		\item $df_z(C^u(z)) \subset C^u(f(z))$ for $z \in K\setminus N^+$, and $df^{-1}_z(C^s(z)) \subset C^s(f^{-1}(z))$ for $z \in f(K \setminus N^+)$; 
		\item for any $n > 0$, we have:
		\begin{align*}
		|df_z^n v| &\geq C\lambda^{n}|v| \quad \textrm{for } z \in K^+, v \in C^u(z); \\
		|df_z^{-n}| &\geq C\lambda^{n}|v| \quad \textrm{for } z \in f^n(K^+), v \in C^s(z).
		\end{align*}
	\end{enumerate}
	In this instance, the set $\Lambda$ defined above is called a \emph{singular hyperbolic attractor}. 
\end{definition}
Define the following subsets of $T_zM$ for $z \in D$: 
\[
E^s_z = \intersection_{n=0}^\infty df^{-n}_{f^n(z)} C^s(f^n(z)) \quad \textrm{and} \quad E^u_z = \intersection_{n = 0}^\infty df^n_{f^{-n}(z))} C^u(f^{-n}(z)). 
\]

\begin{proposition}\cite{PeGHA} The sets $E^s_z$ and $E^u_z$ are subspaces of $T_zM$, called the \emph{stable} and \emph{unstable subspaces} at $z$ respectively. They satisfy the following properties: 
	\begin{enumerate}[label=(\alph*)]
		\item the dimensions of these subspaces are the same as the respective subspaces $P^s_z$ and $P^u_z$ around which the cones $C^s(z)$ and $C^u(z)$ are centered. That is, $\dim E^s_z = \dim P^s_z = p$ and $\dim E^u_z = \dim P^u_z = q = n-p$; 
		\item $T_z M = E^s_z \oplus E^u_z$; 
		\item the angle between $E^s_z$ and $E^u_z$ is bounded below uniformly over $D$; 
		\item for any $n \geq 0$ and $z \in D$, we have
		\begin{align*}
		|df^n_z v| &\leq C\lambda^{-n}|v| \quad \textrm{for } v \in E^s(z), \\
		|df^{-n}_z v| &\leq C\lambda^{-n}|v|\quad \textrm{for } v \in E^u(z). 
		\end{align*}
	\end{enumerate}
\end{proposition}

The distributions $E^s$ and $E^u$ on $D$ thus form uniformly hyperbolic structure with singularities. In particular, they are the tangent spaces of stable and unstable foliations on $D$. To rigorously characterize the leaves of these foliations, we need to define the subsets on which stable and unstable manifolds may be defined. 

For arbitrary $\epsilon > 0$ and $l \in \N$, we denote: 
\begin{align*}
\hat D_{\epsilon, l}^+ &= \left\{z \in K^+ : \rho\left(f^n(z), N^+\right) \geq l^{-1}e^{-\epsilon n}, \: n \geq 0 \right\} ; \\
D^-_{\epsilon,l} &= \left\{ z \in \Lambda : \rho\left(f^{-n}(z), N^-\right) \geq l^{-1}e^{-\epsilon n}, n \geq 0 \right\}; \\
D^+_{\epsilon, l} &= \hat D_{\epsilon, l}^+\cap\Lambda; \\
D^0_{\epsilon, l} &= D^-_{\epsilon, l} \cap D^+_{\epsilon, l}; \\ 
D_{\epsilon}^\pm &= \union_{l \geq 1} D_{\epsilon, l}^\pm; \\
D_\epsilon^0 &= \union_{l\geq 1} D_{\epsilon,l}^0.
\end{align*}
We note that $\hat D^+_{\epsilon, l}$, $D^\pm_{\epsilon, l}$, and $D^0_{\epsilon, l}$ are closed, and hence compact. Also observe that $D^0_\epsilon = D_\epsilon^+\cap D_\epsilon^- \subset D$ for $\epsilon > 0$, and $D_\epsilon^0$ is invariant under both $f$ and $f^{-1}$. Further, $D_\epsilon^+$ and $D_\epsilon^-$ are invariant under $f$ and under $f^{-1}$ respectively. 

For the proof of the following proposition, see the discussion in Sections 1.5 and 2.1 of \cite{PeGHA}. 

\newcommand{\loc}{\mathrm{loc}}

\begin{proposition}\label{(un)stable manifolds}
	There exists $\epsilon > 0$ and such that: 
	\begin{enumerate}[label=(\alph*)]
		\item for $z \in D^+_\epsilon$, there is an embedded (possibly disconnected) submanifold $W_{\mathrm{loc}}^s(z)$ of dimension $p = \dim E^s_z$ for which $T_z W^s_{\mathrm{loc}}(z) = E^s_z$;
		\item for $z \in D^-_\epsilon$, there is an embedded (possibly disconnected) submanifold $W_{\mathrm{loc}}^u(z)$ of dimension $q = n-p = \dim E^u_z$ for which $T_z \Wloc^u(z) = E^u_z$. 
	\end{enumerate}
	Furthermore, define $B^s_z(y,r)$ to be the ball in $W^s_{\mathrm{loc}}(z)$ of radius $r$ centered at $y \in W_\loc^s(z)$, where the distance is the induced distance $\rho^s$ on $W_\loc^s(z)$. Define $B^u_z(y,r)$ and $\rho^u$ similarly. Then there is an $\alpha$ with $\lambda^{-1} < \alpha < 1$ such that for $r > 0$, there is a constant $C = C(r)$ such that: 
	\begin{enumerate}[label=(\alph*)]
		\setcounter{enumi}{2}
		\item for $z \in D^+_\epsilon$, $y \in W_\loc^s(z)$, $w \in B^s_z(y,r)$, and $n \geq 0$, we have 
		\[
		\rho^s(f^n(y), f^n(w)) \leq C \alpha^n \rho^s(y,w); 
		\]
		\item for $z \in D^-_\epsilon$, $y \in W_\loc^u(z)$, $w \in B^u_z(y,r)$, and $n \leq 0$, we have
		\[
		\rho^u(f^n(y), f^n(w)) \leq C \alpha^n \rho^u(y,w). 
		\]
	\end{enumerate}
	Additionally, for $z \in D_{\epsilon,l}^-$, let $B(z,\delta)$ denote the ball of $\rho$-radius $\delta$ centered at $z$. Then there are $\delta_i = \delta_i(z) > 0$, $i=1,2,3$, with $\delta_1 > \delta_2 > \delta_3$, so that for $w \in B(z,\delta_3)$, the intersection $B^s_z(z,\delta_1) \cap W_\loc^u(w)$ is nonempty and contains exactly one point, denoted $[w,z]$; and furthermore, $B^u_{w}([w,z],\delta_2) \subset W_\loc^u(w)$. 
\end{proposition}

We denote
\begin{equation}\label{W^u-def}
W^u(x) = \union_{n \geq 0} f^n\left(\Wloc^u(f^{-n}(x))\right) \quad \textrm{for } x \in K
\end{equation}
and
\begin{equation}\label{W^s-def}
W^s(x) = \union_{n \geq 0} f^{-n}\left(W_\loc^s(f^{n}(x))\cap\Lambda\right) \quad \textrm{for } x \in \Lambda. 
\end{equation}
Given $\delta > 0$ and $x \in K$, let $B_T^u(\delta, x) \subset E^u_x$ denote the open ball of radius $\delta$ in $E^u_x$. For $\delta$ less than the injectivity radius of $M$ at $x$, suppose the connected component of $\left(\exp_x\big|_{B_T^u(\delta,x)}\right)^{-1}(W^u(x)) \subset T_xM$ containing $0$ is the graph of some smooth function $\psi : B_T^u(\delta,x) \to E^s_x$. If such a $\psi$ exists for a particular $x \in M$ and $\delta > 0$, we denote $$W_\delta^u(x) = \exp_x\left(\left\{(u, \psi(u)) : u \in B_T^u(\delta, x) \right\}\right).$$ Such a number $\delta>0$ and such a function $\psi$ exist for each particular $x \in K$ (in particular they form $W^u_\loc(x)$), but $\delta$ may depend on $x$ (and in particular may not have a uniform lower bound). We define $W^s_\delta(x)$ similarly. 

Given local submanifolds $\Wloc^s(z_1)$ and $W_\loc^s(z_2)$, we define the \emph{holonomy map} $\pi: W_\loc^s(z_1) \to W_\loc^s(z_2)$ to be $\pi(w) = [w,z_2] = W_\loc^u(w) \cap W_\loc^s(z_2)$. Let $\nu_{z}^s = \nu|_{W_\loc^s(z)}$ and $\nu_z^u = \nu|_{W_\loc^u(z)}$ denote the induced Riemannian volumes on $W_\loc^s(z)$ and $W_\loc^u(z)$ respectively for $z \in D^\pm_{\epsilon}$.

\begin{proposition}\label{stable-abscont}
	The local foliation $W_\loc^s(z)$ for $z \in D^0_{\epsilon,l}$ is absolutely continuous, in the sense that for any $z_1, z_2 \in D^0_{\epsilon,l}$, the pushforward measure $\pi_* \nu_{z_1}^s$ on $W_\loc^s(z_2)$ is absolutely continuous with respect to $\nu_{z_2}^s$. 
\end{proposition}

\begin{proof}
	This follows from Proposition 10 of \cite{PeGHA}. 
\end{proof}

Generally, maps satisfying (SH1) and (SH2) are dissipative, and so do not preserve Riemannian volume. Therefore, our interest is in the following class of measures: 

\begin{definition}\label{SRB-def}
	A probability measure $\mu$ on $K$ is an \emph{SRB (Sinai-Ruelle-Bowen) measure} if $\mu$ is $f$-invariant and if the conditional measures on the unstable leaves are absolutely continuous with respect to the Riemannian leaf volume. 
\end{definition}

%
%

In \cite{PeGHA}, the existence of SRB measures for singular hyperbolic attractors is proven under certain regulatory conditions. We will describe their construction of SRB measures. Let $J^u(z) = \det \left(df|_{E^u_z}\right)$ denote the unstable Jacobian of $f$ at a point $z \in D$. For $y \in W^u(z)$ and $n \geq 1$, set
\[
\kappa_n(z,y) = \prod_{j=0}^{n-1} \frac{J^u\left(f^{-j}(z)\right)}{J^u\left( f^{-j}(y)\right)}. 
\]
The functions $\kappa_n$ converge pointwise to a function $\kappa$ (see \cite{PeGHA}, Proposition 6(1)). Fix $z \in D^-_\epsilon$ and a sufficiently small $r > 0$, and set
\[
U_0 := B^u(z,r) := B^u_z(z,r), \quad \tilde U_n := f(U_{n-1}), \quad U_n := \tilde U_n \setminus N^+. 
\]
Further set 
\[
\tilde C_0 = 1 \quad \textrm{and} \quad \tilde C_n = \tilde C_n(z) = \left(\prod_{k=0}^{n-1} J^u\left(f^k(z)\right)\right)^{-1}. 
\]
For $n \geq 0$, define the measures $\tilde \nu_n = \tilde \nu_n(z)$ on $U_n$ by
\[
d\tilde\nu_n(y) = \tilde C_n(z) \kappa\left(f^n(z), y\right)d\nu^u_z(y),
\]
and let $\nu_n$ be a measure on $\Lambda$ defined by $\nu_n(A) = \tilde\nu_n(A \cap U_n)$ for any Borel $A \subseteq \Lambda$. Under moderate assumptions, we have that $\nu_n = f^n_*\nu_0$ (see \cite{PeGHA}, Proposition 8). Consider the sequence of measures
\[
\mu_n = \frac 1 n \sum_{k=0}^{n-1}\nu_k.
\]
These measures admit a subsequence that converges in the weak topology to an $f$-invariant SRB measure (\emph{a priori} depending on the reference point $z \in D^-_\epsilon$) on $\Lambda$, proving existence of SRB measures. 

\section{Main result and proof}\label{Main result and proof} In this section, we give conditions under which a singular hyperbolic attractor admits at most finitely many ergodic SRB measures, as well as conditions under which the SRB measure is unique. 

\subsection{Main result}\label{main-result-sec}

We begin by reviewing the assumptions we make on our map $f : K \setminus N \to K$. Assumption (SH1) is the basic setting, and assumption (SH2) concerns the regularity of the map. We now complete the assumptions of our setting. Below, assumption (SH3) concerns the structure of the singular set $N$; assumptions (SH4) - (SH6) concern the smoothness and hyperbolicity of $f$; assumption (SH7) is a further regulatory assumption; and assumption (SH8) is a condition on the regularity of the stable foliation. 

\begin{enumerate}[label=(SH\arabic*),leftmargin=0.5in]
	\setcounter{enumi}{2}
	\item The singular set $N$ is the disjoint union of finitely many embedded submanifolds $N_i$ with boundary, of dimension equal to the codimension of the unstable foliation $W^u$. 
	\item $f$ is continuous and differentiable in $K \setminus N^+$.
	\item $f$ possesses two families of stable and unstable cones $C^s(z)$, $C^u(z)$, for $z \in K \setminus N^+$. 
	\item The assignment $z \mapsto C^u(z)$ has a continuous extension in each $\overline K_i \subset K$ (where $K_i$ are the connected components of $K \setminus N^+$), and there exists $\alpha > 0$ such that for $z \in N \setminus \del K$ and $v \in C^u(z)$, $w \in T_zN$, we have $\angle(v,w) \geq \alpha$. 
	\item $f^j(N^-) \cap N^+ = \emptyset$ for $0 \leq j < k$, where $\lambda^k > 2$ and $$\lambda = \inf_{x \in K \setminus N^+} \norm{df_x} > 1.$$ 
\end{enumerate}

Two general definitions are required before we state assumption (SH8). Let $M$ be a Riemannian manifold, $X \subset M$ a Borel subset, and $\mu$ a measure on $M$ with $\mu(X) > 0$. A partition $\xi$ on $X$ is a \emph{smoothly measurable foliation} if for $\mu$-almost every $x \in X$, the element $\xi_x$ of $\xi$ containing $x$ has the form $\xi_x = W(x) \cap X$, where $W(x)$ is an immersed $C^1$ submanifold in $M$ passing through $X$. Observe, in particular, that the foliations $W^s$ and $W^u$ on $K \setminus N$ as above define smoothly measurable foliations on the attractor $\Lambda$.

Given $x \in M$ and $r > 0$, let $B_r(x)$ denote the ball of radius $r$ in $M$. Consider $X \subset M$ a Borel subset, $\xi$ a smoothly measurable foliation on $X$, and $x \in X$ for which $\xi_x = W(x) \cap X$ for some $C^1$ submanifold $W(x)$. There is a radius $r = r(x)$ for which the submanifold $W(x) \cap B_r(x)$ is the image of a $C^1$ function $\psi_x : U_x \to M$, where $U_x \subset T_x W(x)$ is a neighborhood of $0$ and $\phi_x(0) = x$, $d(\phi_x)_0(T_0 T_x W(x)) = T_xW(x)$. We will say $\xi$ is \emph{locally continuous} if for $\mu$-almost every $x \in X$, the assignments $y \mapsto \phi_y$ and $(y, u) \mapsto d(\phi_y)_u$ are continuous over $y \in X \cap B_{r(x)}(x)$ (note $\phi_y$ is defined $\mu$-almost everywhere in $X \cap B_{r(x)}(x)$), where $u \in U_y \subset T_y W(y)$. 

\begin{enumerate}[label=(SH\arabic*),leftmargin=0.5in]
	\setcounter{enumi}{7}
	\item The stable foliation $W^s$ is locally continuous. 
\end{enumerate}

We now state our main result. 

\begin{theorem}\label{main-result}
	Let $\Lambda$ be a singular hyperbolic attractor of a map $f : K\setminus N \to K$ satisfying conditions (SH1) - (SH7).
	\begin{enumerate}[label=(\alph*)]
		\item There exist at most finitely many ergodic SRB measures of the map $f : \Lambda \to \Lambda$.
		\item If $f$ satisfies condition (SH8), then there exist a collection of pairwise disjoint open sets $U_1, \ldots, U_n$, open in $\Lambda$, for which $\overline {\union_{i = 1}^n U_i}= \Lambda$, and each of which is supported by exactly one ergodic SRB measure. In particular, if $f$ satisfies condition (SH8), then $f|_{\Lambda} : \Lambda \to \Lambda$ is topologically transitive if and only if $f$ admits exactly one ergodic SRB measure. 
	\end{enumerate}
\end{theorem}

\begin{remark}
	The existence of an ergodic SRB measure can be proven under more general assumptions. In particular, in Theorem 1 of \cite{PeGHA}, a regularity condition is given under which a singular hyperbolic attractor admits at least one ergodic SRB measure, and in Theorem 14 of \cite{PeGHA}, it is shown that conditions (SH3) - (SH7) imply this regularity condition. 
\end{remark}

The remainder of this section is devoted to proving this result. In Section \ref{finitely many}, we prove existence of SRB measures and show that a singular hyperbolic attractor is charged by at most finitely many ergodic SRB measures. In Section \ref{transitivity}, we show that $\Lambda$ admits a kind of measurable partition by open sets, each element of which is given full measure by exactly one SRB measure, and thus the SRB measure is unique if (SH8) is satisfied and $f|_{\Lambda}$ is topologically transitive. 


\subsection{Existence of finitely many SRB measures}\label{finitely many} 

The existence of SRB measures for singular hyperbolic attractors follows from the following result, which follows from Theorem 1 and Theorem 14 in \cite{PeGHA}.

\begin{proposition}\label{existence of SRB}
	Suppose $f : K \setminus N \to K$ is a singular hyperbolic map satisfying conditions (SH1) - (SH7). Then there exists an SRB measure for $f$. 
\end{proposition}

We now show that there are only finitely many ergodic SRB measures. We have defined $W^u_\delta(x)$ to be the image under $\exp_x : T_x M \to M$ of the graph of a function $\psi : B^u_T(\delta, x) \to E^s_x$, where $B^u_T(\delta, x) \subset E^u_x$ is the open ball of radius $\delta$ centered at $0 \in E^u_x$, provided that such a function $\psi$ and such a number $\delta>0$ exist. For each $x \in M$, such a $\psi$ and $\delta$ do exist. However, we may also fix $\delta > 0$ and define the set $B^-_\delta$ to be the set of all $x \in D$ for which there is some $\epsilon > 0$, some $l \in \N$, and some $y \in D^-_{\epsilon,l}$ so that $W^u_\delta(y)$ exists and contains $x$. Note that $x \not\in B^-_\delta$ if, for example, $W^u(x)$ is not the image under $\exp_x$ of a smooth graph in a $\delta$-neighborhood of $0 \in T_xM$. 

\begin{proposition}\label{H3}
	For sufficiently small $\epsilon > 0$, we have that $D_\epsilon^0 \neq \emptyset$.
\end{proposition}
\begin{proof}
	This follows from Theorem 14 and Proposition 3 of \cite{PeGHA}.
\end{proof}

Using this proposition, assume $\epsilon > 0$ in $D^\pm_\epsilon$ is chosen so that $D^0_\epsilon \neq \emptyset$. This implies, in particular, that $D^-_\epsilon$ has full measure with respect to any invariant measure. 

Observe that if $\delta_1 < \delta_2$, then $B_{\delta_2}^- \subseteq B_{\delta_1}^-$. Indeed, if $x \in B_{\delta_2}^-$, then $x \in W^u_{\delta_2}(y)$ for some $y \in D_{\epsilon,l}^-$. Since $f$ is regular, $D^-_\epsilon$ has full measure, so $D^-_\epsilon \cap W^u_{\delta_2}(y)$ has full conditional measure. So we can choose $y' \in W^u_{\delta_2}(y)$ with distance $\delta_1$ from $x$ along $W^u_{\delta_2}(y)$, giving us $x \in W^u_{\delta_1}(y')$, so $x \in B^-_{\delta_1}$. 

In particular, if $\mu$ is an ergodic SRB measure that charges $B^-_{\delta_1}$, then since $B^-_{\delta_2} \subset B^-_{\delta_1}$, either $B_{\delta_2}^-$ is charged by $\mu$, or has $\mu$-measure 0. In the latter case, if there is an ergodic SRB measure $\mu_1$ that charges $B^-_{\delta_2}$, then both $\mu_1$ and $\mu$ are ergodic SRB measures charging $B^-_{\delta_1}$. To summarize, if $\delta_2 > \delta_1$, then $B^-_{\delta_1}$ is charged by at least as many, if not more, SRB measures as $B^-_{\delta_2}$. 

Our proof of Theorem \ref{main-result} (a) has two major components. The first is to show that there is a $\delta_0 > 0$ for which $B_{\delta_0}^-$ is charged by every ergodic SRB measure, and so $B_\delta^-$ and $B_{\delta_0}^-$ are charged by the same measures for every $0 < \delta < \delta_0$. The second is to show that every set $B_{\delta}^-$, and in particular $B_{\delta_0}^-$, is charged by only finitely many ergodic SRB measures. 

\begin{lemma}\label{magic delta}
	Suppose the singular hyperbolic map $f : K \setminus N \to K$ satisfies (SH1) - (SH7). There is a $\delta_0 > 0$ such that for every ergodic SRB measure $\mu$ on $\Lambda$, $\mu(B_{\delta_0}^-) >0$. 
\end{lemma}

\begin{proof}
	Assumption (SH3) states that $N$ is composed of finitely many closed submanifolds with boundary. Call these submanifolds $N_i$. Observe that if $U$ is a neighborhood of $N$, then $f(U)$ is a neighborhood of $N^-$. Because $f^j(N^-) \cap N = \emptyset$ for $1 \leq j < k$ with $\lambda^k > 2$ by (SH7), and because $N^-$ and its images are closed (as is $N$), there is a radius $Q > 0$ so that the open neighborhoods $B_Q(N_i)$ of each $N_i$ of radius $Q$ are pairwise disjoint and whose first $k$ images do not intersect $N$. Let $\delta_0 < Q$. 
	
	Fix an ergodic SRB measure $\mu$. Our strategy will be to construct a ``rectangle'' $R \subset \Lambda$, formed by the local hyperbolic product structure of $\Lambda$, with $\mu(R) > 0$. Applying the map $f$ to $R$, the unstable leaves composing $R$ will eventually grow sufficiently large so that a certain iterate of $R$ lies inside $B^-_{\delta_0}$. Since $\mu$ is  $f$-invariant and $\mu(R) > 0$, this will show that $\mu(B^-_{\delta_0}) > 0$ for any ergodic SRB measure $\mu$. 
	
	Proposition \ref{H3} implies $\mu(D^-_\epsilon)>0$. Therefore there is a generic point $x \in D^-_{\epsilon}$, which implies there is an $r > 0$ and an $l \geq 1$ for which $\mu(D^-_{\epsilon,l} \cap B_r(x)) > 0$. By virtue of the hyperbolic local product structure of $D$ (see Proposition \ref{(un)stable manifolds}), there is an $\alpha > 0$ and a $\beta > 0$ for which $W^s_\beta(x) \subset B_r(x)$, and the local unstable leaves $W^u_\alpha(y) \subset B_r(x)$ are well-defined for every $y \in D^-_{\epsilon,l} \cap W^s_\beta(x)$. Furthermore, on a subset $A \subset W^s_\beta(x)$ of full conditional measure, the leaves $W^u_\alpha(y)$ have positive conditional measure for every $y \in A$. Therefore, the set
	\[
	R_1 = \union_{y \in A} W^u_\alpha(y)
	\]
	has positive $\mu$-measure and is contained in $B_r(x)$. 
	
	Suppose $\alpha \geq \delta_0$. Let $y_0 \in A$. Then $W^u_\alpha(y_0)$ is the union of finitely many $W^u_{\delta_0}(y_i)$, with $y_i \in W^u_\alpha(y_0) \cap D^-_\epsilon$. So by definition of $B^-_{\delta_0}$, for $z \in W^u_\alpha(y_0)$, we may take $y = y_i$ in the definition of $B^-_{\delta_0}$ for one of the $y_i$'s, so $z \in B^-_{\delta_0}$. So $W^u_\alpha(y_0) \subset B^-_{\delta_0}$ for every $y_0 \in A$. In particular, $R_1 \subset B^-_{\delta_0}$, so since $R_1$ has positive measure, $\mu(B^-_{\delta_0}) > 0$. 
	
	Now suppose $\alpha < \delta_0$, and let $x_1 = x$. By compactness of $K$, there is a time $j_1 \geq 1$ at which $f^{j_1}(W^u_\alpha(x_1))$ intersects $N$ (and, by assumption, it does so transversally). For $1 \leq j \leq j_1$, the image $f^j(W^u_\alpha(x_1))$ is a local unstable leaf of size $\alpha_j \geq \lambda^j\alpha$. If $\lambda^j\alpha \geq \delta_0$ for some $j \leq j_1$, then using the same arguments as in the previous paragraph, $f^j(W^u_\alpha(x_1)) \subset B^-_{\delta_0}$, for almost every $y_0 \in f^j(W^s_\beta(x_1))$, and more generally, $f^{j}(R) \subseteq B^-_{\delta_0}$, yielding the desired result. 
	
	
	On the other hand, if $\alpha_j < \delta_0$ for $1 \leq j \leq j_1$, then the leaf $f^{j_1}(W^u_\alpha(x_1))$ contains a ball in $\Wloc^u(f^{j_1}(x_1))$ of diameter $\alpha' \geq \frac 1 2 \lambda^{j_1}\alpha$ that does not intersect $N$. Because $\mu$ is an SRB measure, the conditional measure on this ball is absolutely continuous with respect to the Riemannian measure---in particular, the set of generic points in this leaf has full $\mu$-conditional measure. So choose a generic point $x_2$ in this ball. As with $x_1$, we can use the hyperbolic product structure induced by $f$ to create a rectangle $R_2$ of positive $\mu$-measure defined by
	\[
	R_2 = \union_{y \in A_2} W^u_{\alpha'}(y),
	\]
	where $A_2 \subset W^s_{\beta_2}(x_2)$ is a set of full measure for some $\beta_2 > 0$. Relabel $\alpha_{j_1} = \alpha'$ to be the diameter of the local unstable leaf containing $x_2$ and not intersecting $N$. This leaf (and in fact all of $R_2$) lies inside $B_Q(N) = \union_i B_Q(N_i)$, and by our construction of $B_Q(N)$, the first $k$ images of this leaf under $f$ will not intersect $N$. Therefore, either eventually one of the images of this leaf is of diameter $> \delta_0$, or this leaf intersects $N$. In the former case, as before, we have a rectangle of positive $\mu$-measure lying inside of $B^-_{\delta_0}$. In the latter case, the leaf is of diameter $\geq \lambda^k \alpha' \geq \frac 1 2 \lambda^{k+j_1}\alpha$. Again, there is a ball in this leaf of diameter $\geq \frac 1 2 \lambda^k \alpha' \geq \frac 1 4 \lambda^{k+j_1}\alpha$ that does not intersect $N$. As before, we may find a generic point $x_3$ in this ball, and continue iterating this local leaf and resulting positive-measure rectangle. 
	
	
	Proceeding in this way, we construct a sequence of local unstable leaves $W^u_{\alpha_j}(x_m)$, $j = j_1 + \cdots + j_{m-1} + l$, where each leaf image intersects $N$ at time $j_1 + \cdots + j_m$, and thus admits an open ball of sufficient size and not intersecting $N$. By construction, $j_{i+1} - j_i \geq k$ for each $i$, so each leaf has  
	\[
	\alpha_j = \alpha_{j_1 + \ldots + j_{m-1} + l} \geq \frac{\lambda^j}{2^{m-1}}\alpha \geq \left( \frac{\lambda^k}{2}\right)^m \frac{\lambda^{j_1}}{2}\alpha.
	\]
	As $m$ increases, we eventually get $\left(\lambda^k/2\right)^m \lambda^{j_1}\alpha/2 \geq Q > \delta_0$ by (SH7). Once this occurs, we have a rectangle of positive measure contained in $B^-_{\delta_0}$. 
\end{proof}

The next lemma shows that $B_{\delta_0}^-$ is charged by finitely many ergodic SRB measures. Since $B_{\delta_0}^-$ is charged by every ergodic SRB measure, this will prove Theorem \ref{main-result}.

\begin{lemma}\label{finitely many for each delta}
	For sufficiently small $\delta > 0$, the set $B_\delta^-$ admits at most finitely many ergodic SRB measures. In particular, there is a subset $\Lambda^0 \subset B_\delta^-$ that has full measure with respect to any invariant measure, and a finite partition of $\Lambda^0$ each of whose elements is charged by exactly one ergodic SRB measure.\end{lemma}

\begin{proof}
	The proof is an adaptation of a Hopf argument. Define the subsets $\Lambda^+ \subset K$ and $\Lambda^- \subset \Lambda$ respectively to be the set of points where, for every $\phi \in C^0$, the limits
	\[
	\phi_+(x) =  \lim_{n \to \infty} \frac 1 n \sum_{k=0}^{n-1} \phi\left(f^{k}(x)\right) \quad \textrm{and} \quad \phi_-(x) = \lim_{n \to \infty} \frac 1 n \sum_{k=0}^{n-1}\phi\left(f^{-k}(x)\right)
	\]
	exist. By the Birkhoff ergodic theorem, both $\Lambda^+$ and $\Lambda^-$ have full measure with respect to any invariant measure. Observe that if $x \in \Lambda^-$ and $y \in W^u_\alpha(x)$ for $\alpha > 0$, then $\phi_-(y) = \phi_-(x)$, so $y \in \Lambda^-$, and so $W^u_\delta(x) \subseteq \Lambda^-$ for every $x \in \Lambda^-$. Similarly, $W^s_\alpha(x) \subseteq \Lambda^+$ for $x \in \Lambda^+$, $\alpha > 0$. 
	
	Recall that a point $x \in K$ lies in $B_\delta^-$ if and only if there is a $y = y(x) \in D^-_{\epsilon,l}$ for some $\epsilon,l$ for which $x \in W^u_\delta(y)$. So let $\Lambda^0$ be the set of points $x \in B_\delta^-$ for which there is a subset $A \subseteq W^u_\delta(y)$ of full conditional measure (with respect to Lebesgue) such that $A \subseteq \Lambda^+$ and $\phi_+|_A$ is constant for all $\phi \in C^0(K)$. 
	
	We make the following claims: 
	\begin{itemize} 
		\item the set $\Lambda^0$ has full measure in $B^-_\delta$ with respect to any invariant measure, and
		\item the set $\Lambda^0$ is closed.
	\end{itemize}
	Granting these claims for now, using the notation in the above paragraph, for $x \in \Lambda^0$, let $\phi_+(W^u_\textrm{loc}(x)) = \phi_+(z)$ for $z \in A \subseteq W^u_\delta(y)$, where $y = y(x)$ is as in the definition of $B_\delta^-$. We will say that $x, z \in \Lambda^0$ are \emph{equivalent}, and write $x \sim z$, if $\phi_+(W^u_\mathrm{loc}(x)) = \phi_+(W^u_{\mathrm{loc}}(z))$ for every continuous $\phi : K \to \R$. 
	
	Note that the stable foliation $W^s$ is absolutely continuous (see (\ref{W^s-def}) and Proposition \ref{stable-abscont} for the definition of $W^s$ and absolute continuity of the foliation). So if $x \in \Lambda^0$ and $z \in W^s(x) \cap \Lambda^0$, then $\phi_+(x) = \phi_+(z)$, and there is a set $A'$ of full measure in $W^u_\delta(y(z))$ on which $\phi_+$ is constant and equal to $\phi_+(W^u_\mathrm{loc}(x))$. So $x \sim z$ whenever $z \in W^s(x) \cap \Lambda^0$. 
	
	Suppose $\Lambda^0_1$ is an equivalence class, and let $x \in \Lambda_1^0$. We claim there is an $\epsilon > 0$ so that if $y \in \Lambda^0$ lies in the $\epsilon$-ball centered at $x$, then $x \sim y$. It will follow that each equivalence class is open in $\Lambda^0$, and therefore also closed in $\Lambda^0$ since the equivalence classes form a partition of $\Lambda^0$ and $\Lambda^0$ itself is closed. 
	
	To prove this claim, by virtue of Proposition 7 in \cite{PeGHA}, there is an $\epsilon > 0$ for which $B_\epsilon(x)$ has \emph{local hyperbolic product structure:} for $z \in W^u_\epsilon(x)$ and $y \in W^s_\epsilon(x)$, the intersection $W^u_\epsilon(y) \cap W^s_\epsilon(z)$ contains exactly one point, which we denote $[y,z]$. Let $y \in B_\epsilon(x) \cap \Lambda^0$, let $B^u_\epsilon(x)$ denote the ball in the local unstable manifold $W^u_\mathrm{loc}(x)$ centered at $x$ of size $\epsilon$, and let $\theta: B^u_\epsilon(x) \to W^u_\epsilon(y)$ denote the holonomy map $\theta(z) = [y,z]$. 
	
	To show $x \sim y$, let $A \subseteq B_\epsilon^u(x)$ be the set of points $z$ for which $\phi_+(z) = \phi_+(x)$ for every continuous function $\phi$. By definition of the set $\Lambda^0$, the set $A$ has full measure in $B_\epsilon^u(x)$. By absolute continuity of the stable foliation, $\theta(A)$ has full measure in $\theta(B^u_\epsilon(x)) \subset W^u_\epsilon(y)$. Since $\phi_+$ is constant on stable leaves, $\phi_+ \circ \theta = \phi_+$ for every continuous $\phi$. Therefore $\phi_+(z_1) = \phi_+(x)$ for almost every $z_1 \in \theta(B^u_\epsilon(x))$. Again, by definition of $\Lambda^0$, $\phi_+(z_1) = \phi_+(y)$ for every continuous $\phi$ and almost every $z_1 \in \theta(B_\epsilon^u(x))$. Therefore $\phi_+(x) = \phi_+(y)$, and so $x \sim y$. 
	
	It follows from these arguments that each equivalence class is open in $\Lambda^0$, and hence is also closed in $\Lambda^0$. By closedness of $\Lambda^0$ in $K$, there is an $\epsilon > 0$ such that each pair of equivalence classes is separated by a distance of at least $\epsilon$. By compactness of $K$, it follows that there may only be finitely many such equivalence classes. Because an ergodic SRB measure must charge exactly one of these equivalence classes, there may only be finitely many SRB measures. 
	
	It remains only to prove our previous two claims. Let $\hat\Lambda$ denote those points $x \in \Lambda^-$ such that $\phi_-(x) = \phi_+(x)$ for every continuous function $\phi$. By the Birkhoff ergodic theorem, $\hat\Lambda$ has full measure with respect to any invariant measure. Further, let $\hat\Lambda^0$ denote the set of points $x \in \hat\Lambda$ such that there is a set $A \subset W^u_\gamma(x)$ of full conditional measure such that $A \subseteq \Lambda^+$ and $\phi_-(z) = \phi_+(z)$ for every continuous function $\phi$ and all points $z \in A$. Here, $W^u_\gamma(x)$ is the connected component of $W^u(x)$ intersecting $\hat\Lambda$ and containing $x$ (and $W^s_\gamma(x)$ is defined similarly). Since $\phi_-(z)$ takes the same value for all $z \in W^u_\epsilon(x)$, $\phi_-|_A = \phi_+|_A$ is constant. 
	
	
	For $x \in \hat\Lambda^0$, the set $\hat\Lambda^0$ contains the union over $y \in W^s_\gamma(x) \cap B^-_\delta$ of manifolds $W^u_\epsilon(y)$ that contain a subset of full conditional measure (as this subset lies in $\hat\Lambda$, which has full measure). Because $W^s_\gamma(x)$ has full conditional measure in $\hat\Lambda$, it follows that this union also has full measure, from which it follows that $\hat\Lambda^0$ has full measure. 
	
	If a point $x$ has a negative semitrajectory that enters $\hat\Lambda^0$ infinitely often, then one can show $x \in \Lambda^0$. Since $\hat\Lambda^0$ has full measure, the set of points whose negative semitrajectories enter $\hat\Lambda^0$ also has full measure, so $\Lambda^0$ has full measure. This proves the first of the two previous claims. 
	
	To show that $\Lambda^0$ is closed, suppose $x$ is the limit point of a sequence $(x_n)$ in $\Lambda^0$. Since the stable foliation is absolutely continuous and $W^u_\epsilon(x_n)$ converges to $W^u_\epsilon(x)$ for $\epsilon > 0$ small, we can find a set $A \subset W^u_\delta(y)$ of full conditional measure, where $y = y(x)$ is as in the definition of $B_\delta^-$, such on which $\phi_+$ is constant for all continuous functions $\phi$. Therefore $x \in \Lambda^0$. 
\end{proof}


\begin{proof}[Proof of Theorem \ref{main-result}(a)] Note $B^- := \union_{\delta > 0} B^-_\delta$ is invariant under $f$, and as we showed in Lemma \ref{magic delta}, $\mu(B^-) > 0$ for every ergodic SRB measure $\mu$. Therefore $\mu(B^-) = 1$ for every ergodic SRB measure $\mu$, and since $\mu(D)=1$ as well, every ergodic SRB measure gives full volume to $D \cap B^-$. If there were infinitely many ergodic SRB measures, then by Lemma \ref{magic delta}, there would be a $\delta > 0$ for which $B^-_\delta$ is charged by infinitely many SRB measures. But this contradicts Lemma \ref{finitely many for each delta}.
\end{proof}

\subsection{Uniqueness and topological transitivity}\label{transitivity}

Generally speaking, a map satisfying (SH1) - (SH7) may admit more than one ergodic SRB measure (see examples in Section \ref{section-examples}). However, under moderate assumptions on the regularity of the stable foliation $W^s$, one can show that the components of topological transitivity of $f|_{\Lambda} : \Lambda \to \Lambda$ are in correspondence with the number of distinct ergodic SRB measures. We formalize this idea in this section. 

%

Given a metric space $X$, we will call a Borel measure $\mu$ on $X$ \emph{locally positive} if $\mu(U) > 0$ for every nonempty open subset $U \subset X$. A collection $\{U_i\}_{i \in I}$ of open subsets $U_i \subset X$, together with a collection of Borel measures $\{\mu_j\}_{j \in J}$ on $X$, with $J \subset I$, is an \emph{open partition by measures} if: 
\begin{enumerate}[label=(P\arabic*)]
	\item the open sets $\{U_i\}_{i \in I}$ are pairwise disjoint;
	\item $\overline{\union_{i \in I} U_i} = X$;
	\item $\mu_j \left(X \setminus U_j\right) = 0$ for all $j \in J$; and
	\item $\mu_j|_{U_j}$ is locally positive for all $j$. 
\end{enumerate}
By local positivity, we may assume $I \setminus J$ has a single element, which we denote $0 \in I$. If the open sets $\{U_i\}_{i \in I}$ and the measures $\{\mu_j\}_{j \in J}$ are chosen so that $I = J$, we say that the open partition by measures is \emph{complete}. If $f : X \to X$ is continuous, and the measures $\mu_j$ are ergodic probability measures, we call an open partition by the ergodic measures $\mu_j$ an \emph{open ergodic partition} if in addition to (P1) - (P4) above, we also have
\begin{enumerate}[label=(P\arabic*)]
	\setcounter{enumi}{4}
	\item $\mu(U_0) = 0$ for any ergodic measure $\mu$. 
\end{enumerate}

\begin{lemma}\label{erg-comps}
	If $X$ is a complete metric space, and $f : X \to X$ is a continuous and open map admitting an open ergodic partition by the open sets $\{U_i\}_{i \in I}$ and the ergodic $f$-invariant Borel measures $\{\mu_j\}_{j \in I \setminus \{0\}}$, then $\overline U_i$ is $f$-invariant for each $i$. If the open ergodic partition is complete, then $f|_{\overline U_i}$ is topologically transitive for each $i \in I$. 
\end{lemma}

\begin{proof}
	By (P1) - (P4) and the fact that each measure $\mu_j$ is $f$-invariant, $U_i$ is invariant for every $i$, as is $\overline{U}_i$. Now suppose the open ergodic partition is complete. If $V, V' \subset \overline U_i$ are open, then $\union_{n \in \Z} f(V)$ is open and $f$-invariant. Thus if $f^n(V) \cap V' = \emptyset$ for every $n \geq 0$, ergodicity of $\mu_i$ implies that either $\mu_i(V') = 0$ or $\mu_i(V) = 0$. By local positivity, either $V' = \emptyset$ or $V = \emptyset$.
\end{proof}

Theorem \ref{main-result}(b) is now an immediate consequence of the following lemma. 

\begin{lemma}\label{SRB-uniqueness}
	If $f : K \subset N \to K$ is a singular hyperbolic map with attractor $\Lambda \subset K$ satisfying condition (SH8) in addition to (SH1) - (SH7), then $\Lambda$ admits a finite and complete open ergodic partition, and the measures defining this partition are SRB measures. If in addition $f|_\Lambda : \Lambda \to \Lambda$ is topologically transitive, then the open ergodic partition consists of a single open set and a single measure, and thus in particular, $f$ admits a unique ergodic SRB measure. 
\end{lemma}

\begin{proof}
	The fact that $\Lambda$ admits a complete open ergodic partition by SRB measures follows from Theorems 4 and 6 of \cite{PeGHA}. Finiteness of this partition follows from Theorem \ref{main-result}(a). If $f|_{\Lambda}$ is topologically transitive, by Lemma \ref{erg-comps}, there is at most one ergodic SRB measure; existence of this measure follows from Proposition \ref{existence of SRB}. 
\end{proof}

\section{Examples}\label{section-examples}

Maps described in Theorem \ref{main-result} do exist, and as the following non-example will demonstrate, the hypotheses described in this result are necessary assumptions in general. Examples of singular hyperbolic attractors can be found in, for example, Lorenz-type attractors; but this class of singular hyperbolic maps includes cases where the singular set has countably many components, and admit countably many SRB measures. 

\subsection{Lorenz-type attractors}

To begin, we describe the class of singular hyperbolic attractors generated by Lorenz-type maps of the unit square. The definition of these maps is as follows. Let $I = (-1,1)$, $K = I^2$, and $-1 = a_0 < a_1 < \cdots < a_m < a_{m+1} = 1$. Define the rectangles $P_i = I \times (a_i, a_{i+1})$ for $0 \leq i \leq m$, and $N = I \times \{a_0, \ldots, a_{m+1}\}$. Let $f : K \setminus N \to K$ be an injective map given by 
\[
f(x,y) = \big( \phi(x,y), \, \psi(x,y)\big), \quad x, y \in I, 
\]
where the functions $\phi, \psi : K \to \R$ satisfy the following conditions: 
\begin{enumerate}[label=(L\arabic*)]
	\item $\phi$ and $\psi$ are continuous in $\overline P_i$ for each $i$, and:
	\begin{align*}
	\lim_{y \to a_i^+} \phi(x,y) = \phi_i^+, &\quad \lim_{y \to a_i^+} \psi(x,y) = \psi_i^+, \\
	\lim_{y \to a_i^-} \phi(x,y) = \phi_i^-, &\quad \lim_{y \to a_i^-} \psi(x,y) = \psi_i^-,
	\end{align*}
	where $\phi_i^{\pm}$, $\psi_i^\pm$ do not depend on $x$ for each $i$; 
	\item $\psi$ and $\phi$ have two continuous derivatives in $P_i$. Furthermore, there are positive constants $B_i^1$, $B_i^2$, $C_i^1$, and $C_i^2$; constants $0 \leq \nu_i^1, \nu_i^2, \nu_i^3, \nu_i^4 \leq 1$; a sufficiently small constant $\gamma > 0$; and continuous functions $A_i^1(x,y)$, $A_i^2(x,y)$, $D_i^1(x,y)$, and $D_i^2(x,y)$ that tend to zero uniformly over $x$ as $y \to a_i$ or $y \to a_{i+1}$; so that for $(x,y) \in P_i$, 
	\[
	\left.\begin{array}{l}
	d\phi(x,y) = B_i^1(y - a_i)^{-\nu_i^1}\left(1+A_i^1(x,y)\right) \\ d\psi(x,y) = C_i^1(y-a_i)^{-\nu_i^2}\left(1+D_i^1(x,y)\right)
	\end{array}\right\} \quad \textrm{if } y-a_i \leq \gamma;
	\]
	\[
	\left.\begin{array}{l}
	d\phi(x,y) = B_i^2(a_{i+1} - y)^{-\nu_i^3}\left(1+A_i^2(x,y)\right) \\ d\psi(x,y) = C_i^2(a_{i+1}-y)^{-\nu_i^4}\left(1+D_i^2(x,y)\right)
	\end{array}\right\} \quad \textrm{if } a_{i+1} - y \leq \gamma;
	\]
	and additionally, $\norm{\phi_{xx}}, \norm{\psi_{xx}}, \norm{\phi_{xy}}, \norm{\psi_{xy}} \leq \mathrm{const.}$; 
	\item the following inequalities hold: 
	\begin{align*}
	\norm{f_x}, \norm{g_y^{-1}} &< 1; \\
	1-\norm{g_y^{-1}}\norm{f_x} &> 2\sqrt{\norm{g_y^{-1}}\norm{g_x}\norm{g_y^{-1}f_y}}; \\
	\norm{g_y^{-1}} \norm{g_x} &< \left(1-\norm{f_x}\right)\left(1-\norm{g_y^{-1}}\right);
	\end{align*}
	where $\norm{\cdot} = \max_i \sup_{(x,y) \in P_i} |\cdot|$. 
\end{enumerate}

This class of maps includes the \emph{geometric Lorenz attractor}, for which we have $m=1$, $a_1 = 0$, and
\begin{align*}
\phi(x,y) &= \left( -B|y|^{\nu_0} + Bx\,\sgn(y)|y|^\nu+1\right) \sgn(y), \\
\psi(x,y) &= \left((1+A)|y|^{\nu_0} - A\right)\sgn(y),
\end{align*}
where $0 < A < 1$, $0 < B < \frac 1 2$, $1/(1+A) < \nu_0 < 1$, and $\nu > 1$. 

\begin{theorem}\label{Lorenz-type theorem}
	Let $f : I^2 \setminus N \to I^2$ be a Lorenz-type map for which one of the following properties hold: 
	\begin{enumerate}[label=(\alph*)]
		\item $\nu_i^j = 0$, for $i = 1, \ldots, m$ and $j = 1, 2, 3, 4$; 
		\item $\rho\left(f^n(\phi_i^\pm, \psi_i^\pm), N\right) \geq C_i e^{-\gamma n}$ (where $C_i$ are constants independent of $n$ and $\gamma > 0$ is sufficiently small).
	\end{enumerate}
	Then $f$ admits a singular hyperbolic attractor $\Lambda$, which is supported by at most finitely many ergodic SRB measures. 
\end{theorem}

\begin{remark}
	This result is also proven in \cite{APPV}, and is also a consequence of the arguments in both \cite{Asectional} and \cite{Sat92}. We present an additional proof of this result using Theorem \ref{main-result} directly. 
\end{remark}

\begin{proof}
	Properties (SH1) and (SH4)-(SH7) are shown in \cite{PeGHA}, Theorem 17. The singular set $N$ is the disjoint union of finitely many horizontal lines $I \times \{a_i\}$, $i = 1, \ldots, m$, so (SH3) is satisfied. The statement now follows from Theorem \ref{main-result}.
\end{proof}

\begin{remark}
	Condition (SH4) is easy to verify for the geometric Lorenz attractor, as the map $\phi : I^2 \setminus \left(I\times\{0\}\right) \to \R$ extends to $\pm 1$ as $y \to 0$ from above or below. In particular, $N^- \cap K = \emptyset$, since the continuations of $\phi$ to $N$ map $N$ to the boundary of $K$, so (SH4) is in fact trivial. Moreover, this is true with any Lorenz-type attractor for which $\phi_i^\pm = \pm 1$ or $\mp 1$. 
	
	More generally, (SH4) holds if (b) is satisfied in the statement of Theorem \ref{Lorenz-type theorem}, provided $\gamma>0$ is sufficiently small. 
\end{remark}

\subsection{The Belykh attractor} We consider a map $f : K \setminus N \to K$, where $K = [-1,1]^2$, and
\[
N = \{(x, kx) \in K : -1 < x < 1\}
\]
where $|k| < 1$. (More generally one can consider $N = \{(x, h(x)) : -1 < x < 1\}$ for a continuous function $h$.) We then choose constants $\lambda_1, \lambda_2, \mu_1, \mu_2$ so that
\[
0 < \lambda_1, \mu_1 < \frac 1 2 \quad \textrm{and} \quad 1 < \lambda_2, \mu_2 < \frac{2}{1+|k|},
\]
and define the map $T : K \setminus N \to \R^2$ by
\[
T(x,y) = \begin{cases}
\left( \lambda_1(x-1)+1, \: \lambda_2(y-1)+1\right) & \textrm{if } y > kx; \\
\left( \mu_1(x+1)-1,\: \mu_2(y+1)-1\right) &\textrm{if } y < kx.
\end{cases}
\]
This map was first introduced in \cite{Belykh} as a model of phase synchronization theory. 

\begin{theorem} Define $T : K \setminus N \to \R^2$ as above. 
	\begin{enumerate}[label=(\alph*)]
		\item The map $T$ is a map from $K\setminus N$ into $K$, and satisfies conditions (SH1) and (SH3)-(SH6). 
		\item For any choice of $\lambda_2 > 1$, and for all but countably many $\mu_2 > 1$ (the countably many exceptions depending on $\lambda_2$), there is a $\delta > 0$ so that $T$ satisfies (SH7) when $|k| < \delta$, and thus admits finitely many ergodic SRB measures for such $k$. 
	\end{enumerate}
\end{theorem}

\begin{proof}
	The first of the above assertions is proven in \cite{PeGHA}. To prove the second, first note that \cite{PeGHA} shows that $T$ satisfying (SH3)-(SH6) admits countably many SRB measures. In general, (SH7) may fail; however, we will show that given $\lambda_2 > 0$, this can happen only for countably many choices of $\mu_2>0$. To see this, note that when $k=0$, (SH7) fails if the horizontal lines forming $N^-$ lie inside $f^{-n}(N)$ for some $n > 0$, which only happens for countably many choices of $\mu_2$. Given a pair $\lambda_2$ and $\mu_2$ so that $T$ satisfies (SH7) with $k=0$, the line segments forming $N$ and $f^j(N^-)$ do not intersect for $0 \leq j < k$, where $\max(\lambda_2, \mu_2)^k > 2$. Increasing $|k|$ will rotate these line segments; by continuity, if the change in $|k|$ is sufficiently small, these line segments will remain disjoint. So, for these choices of $\lambda_2$, $\mu_2$, and $k$, $T$ will admit finitely many SRB measures by Theorem \ref{main-result}. 
\end{proof}

\subsection{Necessity of assumptions}\label{infinite-example} The singular set $N$ may in principle have countably many components. If this is the case, then the attractor may admit infinitely many ergodic SRB measures, as the following example illustrates. 

Let $P_k = \left(-1,1\right) \times \left(2^{-k}-1, \, 2^{-(k-1)}-1\right)$ for $k \geq 0$. Then $K = I^2 = \overline{\union_k P_k}$, and $ N^1 := K \setminus \union_k P_k$ is the countable union of line segments $(-1,1) \times \{2^{-k}-1\} =: N_k^1$. 

Let $f : I^2\setminus \big( (-1,1) \times \{0\} \big)\to I_2$ be the geometric Lorenz attractor, and let $f_k : P_k \setminus \big( (-1,1) \times \left\{\frac{2^{-k-1} + 2^{-k}}{2}-1\right\}$ be given by $f_k = h_k^{-1} \circ f \circ h_k$, where $h_k : P_k \to I^2$ is the conjugacy map given by 
\[
h_k(x,y) = \left(x, 2^{k+2}(y+1)-3\right). 
\]
Now denote
\begin{align*}
N &= I^2 \setminus \left( P_k \setminus \left( (-1,1) \times \left\{\frac{2^{-k-1} + 2^{-k}}{2}-1\right\}\right)\right) \\
&= (-1,1) \times \union_{k \geq 0} \left( \left\{ 2^{-k}-1, \frac{2^{-k}+2^{-k-1}}2 - 1\right\}\right),
\end{align*}
and let $g : I^2 \setminus N \to I^2$ be given by 
\[
g(x,y) = f_k(x,y) \quad \textrm{for } (x,y) \in P_k \setminus \left( (-1,1) \times \left\{\frac{2^{-k-1} + 2^{-k}}{2}-1\right\}\right).
\]
Effectively, we have embedded the geometric Lorenz attractor into each disjoint rectangle $P_k$. The map $g$ admits a singular hyperbolic attractor, and the singular set $N$ is the disjoint union of countably many submanifolds. Since each orbit of $g$ is entirely contained in one of the rectangles $P_k$, each $P_k$ supports a distinct ergodic SRB measure. So the requirement that there are only finitely many components of the singular set $N$ is a necessary assumption for our result to hold. 

\section*{Acknowledgments} I would like to thank Penn State University and the Anatole Katok Center for Dynamical Systems and Geometry where this work was done. I also thank my advisor, Y. Pesin, for introducing me to this problem and for valuable input over the course of my investigation into singular hyperbolic attractors. I also thank S. Luzzatto for many helpful remarks concerning the connections between ergodicity and topological transitivity of singular hyperbolic attractors. 


\medskip
Received xxxx 20xx; revised xxxx 20xx.
\medskip

\end{document}